\documentclass[12pt,leqno]{article}

\usepackage{palatino,mathrsfs}
\usepackage[mathcal]{euler}

\usepackage{xcolor}

\usepackage{graphicx} 
\usepackage{epstopdf,epsfig}

\usepackage{comment} 
\usepackage{amsmath,amsthm,amsfonts,amssymb,latexsym,amscd,enumerate,url,hyperref}
\usepackage{amssymb}
\usepackage{url}
\usepackage{graphicx}
\usepackage[all,knot]{xy}
\xyoption{arc}

\usepackage{chngcntr}
 \counterwithin{equation}{section}

\usepackage{multirow}

\usepackage{caption}

\theoremstyle{plain}
\newtheorem*{theorem}{Theorem}
\newtheorem*{proposition}{Proposition}

\theoremstyle{definition}
\newtheorem*{defn}{Definition}

\theoremstyle{remark}
\newtheorem*{remark}{Remark}

\def\R{\mathbb R}
\def\C{\mathbb C}

\begin{document}

\title{Contractions of Representations and Algebraic Families of Harish-Chandra Modules}
 
\author{Joseph Bernstein\thanks{School of Mathematical Sciences, Tel Aviv University, Tel Aviv 69978,
Israel.} , \ Nigel Higson\thanks{Department of Mathematics, Penn State University, University Park, PA 16802, USA.} \ and Eyal Subag\footnotemark[2]}

 
 

\date{}

\maketitle


\begin{abstract}
\noindent 
We  examine from an algebraic point of view some families of unitary group representations that arise in mathematical physics and are  associated to  contraction families of Lie groups.  The contraction families  of groups    relate different real forms of a reductive group and  are continuously parametrized, but the unitary representations are defined over a parameter subspace   that  includes both discrete and continuous parts. Both finite- and infinite-dimensional representations can occur, even within the same family.
We shall  study the simplest nontrivial examples, and use the concepts of algebraic families of Harish-Chandra pairs and Harish-Chandra modules, introduced in a previous paper, together with the Jantzen filtration, to construct these    families of unitary representations algebraically. 
\end{abstract}

\setcounter{tocdepth}{2}


\section{Introduction} 
In a previous paper \cite{Ber2016} we showed how  certain contraction families of Lie groups that arise in mathematical physics  \cite{Segal51,Inonu-Wigner53}  can be constructed    as real points of   algebraic families of complex algebraic groups.  In this paper we shall examine associated  families of unitary representations, and show how they  can be obtained from  algebraic families of Harish-Chandra modules. 

The focus of our study will be a family of groups obtained from  $SU(1,1)$ and its Cartan decomposition
\[
\mathfrak{su}(1,1) = \mathfrak{u}(1)\oplus \mathfrak p,
\]
where $\mathfrak {u}(1)$ is realized as the diagonal matrices in $\mathfrak{su}(1,1)$, and $\mathfrak{p}$ is the vector space of matrices in $\mathfrak{su}(1,1) $ with zero diagonal entries.
The  group  $SU(1,1)$ may be contracted to its Cartan motion group $U(1)\ltimes \mathfrak p$. This means that  a smooth family of Lie groups $\{G_t\}_{t\in \R}$ may be constructed  with 
\[
G_t\cong
\begin{cases}  SU(1,1) &   t\neq 0 \\ U(1)\ltimes \mathfrak p & t = 0 .
\end{cases}
\]
 There is also a very similar  contraction of $SU(2)$ to the same motion group.  For all this see  \cite{Dooley85,Dooley83}. 
 
Every infinite-dimensional  unitary irreducible representation of the motion group  $U(1)\ltimes \mathfrak p$ can be approximated by two different families of representations associated to these two contractions.
The first is a continuous family of unitary principal series representations of $SU(1,1)$ and the second is a discrete family of finite-dimensional irreducible representations of $SU(2)$.  These two families of representations  are examples of the two different procedures for contraction of representations that were introduced in \cite[Sections 2(a) and 2(b)]{Inonu-Wigner53}.  

The problem that we shall address in this paper is to understand these families algebraically. 
We shall show that   the two  families of representations  may be obtained  from  one algebraic family of Harish-Chandra  modules by means of a single procedure that uses   Jantzen filtration techniques. 

In \cite{Ber2016} we constructed an  algebraic family of Harish-Chandra pairs  $(\boldsymbol{\mathfrak{g}},{\boldsymbol{K}})$   over the complex affine line\footnote{Actually the family extends to a family over the projective line, but the phenomena that we shall study in this paper are purely local, and so we shall work over the affine line.} with fibers   
\begin{eqnarray*}
&&\bigl (\boldsymbol{\mathfrak{g}}|_z,\boldsymbol{K}|_z\bigr )\cong \begin{cases}
\bigl (\mathfrak{su}(1,1)_\C,U(1)_\C\bigr )&z\neq 0\\
\bigl (\mathfrak{u}(1)_\C \ltimes \mathfrak{p}_\C, U(1)_\C\bigr )& z=0 .
\end{cases} 
\end{eqnarray*}
It may be equipped with a real structure $\sigma$, and there is a corresponding family of real groups over the real line with  fibers  
\begin{eqnarray*}
&&\boldsymbol{G}^{\sigma}|_{t}\cong \begin{cases}
SU(1,1)&t>0\\
U(1)\ltimes \R^2& t=0\\
SU(2)& t<0 .
\end{cases} 
\end{eqnarray*}
We shall review these constructions in Sections~\ref{sec-algebraic-families} and \ref{sec-su(1,1)-family}.  The family $\boldsymbol{G}^{\sigma}$ combines the two contractions mentioned above.

Let $\mathcal{F}$ be   a  quasi-admissible and    generically irreducible  algebraic family of Harish-Chandra modules   for  $(\boldsymbol{\mathfrak{g}},{\boldsymbol{K}})$, and assume that $\mathcal{F}$ is rationally isomorphic to its $\sigma$-twisted dual $\mathcal{F}^{\langle\sigma\rangle}$; see Section~\ref{sec-twist}.  We shall   obtain from  $\mathcal{F}$  a family  of unitary representations of  a subfamily of $\boldsymbol{G}^{\sigma}$.  

Our method is to  associate to an   intertwining operator  from  $\mathcal{F}$  to   $\mathcal{F}^{\langle\sigma\rangle}$  (defined over the field of rational functions) Janzten-type filtrations of the fibers  of $\mathcal{F}$. Jantzen's technique equips the subquotients with nondegenerate hermitian forms.   If we isolate those subquotients on which the hermitian forms are  definite, then the Harish-Chandra module structures on these fibers may be integrated to unitary group representations.  

The generically irreducible,  algebraic families of Harish-Chandra modules  for the family  $(\boldsymbol{\mathfrak{g}},{\boldsymbol{K}})$ were analyzed in \cite{Ber2016}.  Applying the above method to a family of ``spherical principal series type'' we obtain precisely the family of unitary representations described above, consisting of the family of all unitary spherical principal series representations for $SU(1,1)$ when $t>0$, together with highest weight $2m$    spherical unitary irreducible representations of $SU(2)$ when $t=-1/(2m+1)^2$.  (In addition we obtain the unitary irreducible representation of the motion group at $t=0$ to which both of these families of representations converge.)

In summary, we are able to  place  the contraction families  of unitary representations within a purely algebraic context.  This is despite the fact that the ``spectrum'' of values $t$ over which the contraction families are defined includes both discrete and continuous parts, together. 

In fact the structure of this spectrum strongly recalls  the quantum mechanics of the hydrogen atom. It will be shown elsewhere that the  family of Harish-Chandra pairs studied in this paper arises as    symmetries of  the Schr\"{o}dinger equation for the hydrogen atom, and the collection of all physical solutions of the Schr\"{o}dinger equation coincides with  one of the families of representations studied here. This suggests that our techniques for obtaining unitary  representations   from   algebraic families of Harish-Chandra modules  may be useful for  quantization.

Another application that we shall explore elsewhere is  the Mackey bijection.  Mackey showed in some specific cases how to place most of the unitary dual of  a  semisimple Lie group   in bijection with most of the unitary dual of the corresponding Cartan motion group  \cite{Mackey75}.  Later on a precise bijection between the \emph{tempered} duals was obtained---first for complex groups in  \cite{Higson08}, and then for real groups in \cite{Afgoustidis15}. The bijection for complex groups was examined from an algebraic perspective in  \cite{Higson11}, and extended to admissible duals. This latter bijection fits very well with the methods of this paper, although, since unitarity is not the main issue for the Mackey bijection, we shall identify appropriate subquotients in this context using minimal $K$-types rather than definiteness of hermitian forms.

 Finally, perhaps it is also worth mentioning that the algebraic  families of Harish-Chandra pairs considered here, which place    real reductive pairs together with their compact forms,  provide a means to study  hermitian forms on Harish-Chandra modules and  $c$-invariant forms in the sense of  \cite[Section 10]{AdamsEtAl2015}  within one algebraic context.  See also \cite{Adams2017} for another study of families of Harish-Chandra modules inspired by similar considerations.

The first and third authors were partially supported by ERC grant 291612. Part of their work on this project was done at Max-Planck Institute for Mathematics, Bonn, and they would like to thank MPIM for the very stimulating atmosphere. The second author was partially supported by NSF grant DMS-1101382.

\section{Algebraic families and real structures}
\label{sec-algebraic-families}
In this section we fix notations and quickly recall some   definitions that were given in \cite{Ber2016} (we refer the reader to that paper for more details). 
Throughout,  by a \emph{variety} we shall mean an irreducible, nonsingular, quasi-projective, complex algebraic variety (in all the examples and computations done in the paper, the variety will be the complex affine line).   If $X$ is a variety, then we shall  denote its structure sheaf by $O_X$.

\subsection{Algebraic families of Harish-Chandra  pairs}
An \emph{algebraic family of complex Lie algebras} over a variety $X$ is    a locally free sheaf of $O_X$-modules   that is equipped with an $O_X$-linear Lie bracket to make it a sheaf of Lie algebras.  An \emph{algebraic family of complex algebraic groups} over $X$ is a smooth group scheme over $X$. Thanks to the smoothness assumption,   to any algebraic family of complex algebraic groups   there is a corresponding algebraic family of complex Lie algebras.

Suppose given an algebraic family  ${\boldsymbol{K}}$ of complex algebraic groups and an  algebraic family $\boldsymbol{\mathfrak{g}}$ of complex Lie algebras   on which ${\boldsymbol{K}}$ acts, along with  a ${\boldsymbol{K}}$-equivariant embedding of families of Lie algebras $j:\operatorname{Lie}(\boldsymbol{K})\longrightarrow \boldsymbol{\mathfrak{g}}$.  This data defines an \emph{algebraic family of Harish-Chandra pairs} if the two actions of $\operatorname{Lie}(\boldsymbol{K})$ on $\boldsymbol{\mathfrak{g}}$ coincide.

We shall  deal in this paper with algebraic families of Harish-Chandra pairs $(\boldsymbol{\mathfrak{g}},{\boldsymbol{K}})$ with a constant group scheme whose fiber is  connected and reductive. That is,  we shall deal here only with cases where $\boldsymbol{K}=X{\times} K$ for  some  complex connected reductive group $K$. 

\subsection{Algebraic families of Harish-Chandra  modules}
 Let $(\boldsymbol{\mathfrak{g}},{\boldsymbol{K}})$ be an algebraic family of   Harish-Chandra pairs over $X$. An \emph{algebraic family of Harish-Chandra modules} for  $(\boldsymbol{\mathfrak{g}},{\boldsymbol{K}})$ is a flat, quasicoherent  $O_X$-module $\mathcal{F}$  that is equipped with
\begin{enumerate}[\rm (a)]

\item an  action of $\boldsymbol{K}$ on $\mathcal{F}$, and

\item an  action  of  $\boldsymbol{\mathfrak{g}}$ on $\mathcal{F}$,

\end{enumerate}
such that  the action morphism
\[
 \boldsymbol{\mathfrak{g}}\otimes_{O_X}\mathcal{F}\longrightarrow \mathcal{F}
\]
 is $\boldsymbol{K}$-equivariant, and such that the differential of the $\boldsymbol{K}$-action in (a) is equal to the composition of the embedding of $\operatorname{Lie}(\boldsymbol{K})$ into $\boldsymbol{\mathfrak{g}}$   with the action of   $\boldsymbol{\mathfrak{g}}$ on $\mathcal{F}$. 

Assuming, as indicated above, that the  family $\boldsymbol{K}$ is constant and reductive,  $\mathcal{F}$  is said to be    \emph{quasi-admissible}  if 
$[ \mathcal{L}\otimes_{O_X} \mathcal{F}]^{\boldsymbol{K}}$ is a locally free and finitely generated sheaf of $O_X$-modules  for every  family $\mathcal{L}$ of representations of $\boldsymbol{K}$ that is locally free and   finitely generated  as an $O_X$-module.
In this case, there is a canonical isotypical decomposition  
   \[
   \mathcal F = \bigoplus _{\tau\in \widehat K}   \mathcal{F}_\tau .
   \]

\subsection{Real structures}

 Denote by $\overline{X}$  the complex conjugate variety of $X$. A \emph{real structure} on $X$  is a morphism of varieties   $\sigma_X:X\longrightarrow \overline{X}$ such that  the composition 
  $$\xymatrix{
X\ar[r]^{\sigma_X} & \overline{X}\ar[r]^-{\overline{\sigma_X}} & X} $$
 is the identity morphism (here $\overline{\sigma_X}$ is equal to $\sigma_X$ as a map of sets; it is a morphism of varieties from $\overline{X}$ to $X$). Denote by $X^{\sigma}\subseteq X$ the set of points  that are  fixed by $\sigma$ (recall that  $X$ and $\overline{X}$ are equal as sets).   The fixed set might be empty.
 
 If $\mathcal{F}$ is a sheaf of $O_X$-modules, then its complex conjugate $\overline {\mathcal{F}}$  is a sheaf of $O_{\overline X}$-modules.  Given a real structure on $X$, a (compatible)  real structure on $\mathcal{F}$ is a morphism
 \[
 \sigma_\mathcal{F}:\mathcal{F}\longrightarrow \sigma_X^*\overline{\mathcal{F}} 
   \]
  of $O_X$-modules  for which the composition 
\[
\xymatrix@C=30pt{
\mathcal{F}   \ar[r]^-{\sigma_{\mathcal{F}}} &
\sigma_X^*\overline{\mathcal{F}}   \ar[r]^-{\sigma_X^* ( \overline{\sigma_{\mathcal{F}}}) } &
\sigma_X^* \overline{\sigma_X^*\overline{\mathcal{F}}}  \ar[r]^-\cong & \mathcal{F}} 
\]
 is the identity morphism.

Finally, let  $(\boldsymbol{\mathfrak{g}},{\boldsymbol{K}})$  be an algebraic family of Harish-Chandra pairs over $X$.
A \emph{real structure on $(\boldsymbol{\mathfrak{g}},{\boldsymbol{K}})$} is a triplet $ \{\sigma_X,\sigma_{\boldsymbol{\mathfrak{g}}},\sigma_{\boldsymbol{K}}\}$  with $\sigma_X$ a real structure on $X$, $\sigma_{\boldsymbol{\mathfrak{g}}}$ a compatible real structure on the  $O_X$-module $\mathfrak{g}$, and  $\sigma_{\boldsymbol{K}}$ a real structure on  $\boldsymbol{K}$ for which  the morphisms
\begin{eqnarray*}
&&\sigma_{\boldsymbol{\mathfrak{g}}}:\boldsymbol{\mathfrak{g}}\longrightarrow \sigma_X^*\overline{\boldsymbol{\mathfrak{g}}}\\
&&\sigma_{\boldsymbol{K}} :\boldsymbol{K}\longrightarrow \sigma_X^*\overline{\boldsymbol{K}}
\end{eqnarray*} 
are a morphism of algebraic families of Harish-Chandra pairs over $X$.
 
 For  further discussions about real structures see for example   \cite[Chapter 1]{Borel1991} or \cite[Chapter 11]{Springer}.

 \subsection{The   sigma-twisted dual}
 \label{sec-twist}
 If $(\mathfrak{g}, K)$ is an ordinary Harish-Chandra pair (not a family) and if $V$ is an admissible $(\mathfrak{g}, K)$-module, then its \emph{contragredient}    $V^{\dagger}$ is the admissible $(\mathfrak{g}, K)$-module consisting of all $K$-finite  linear functionals on $V$.  The   conjugate contragredient   $\overline{V^{\dagger}}$  is a   $(\overline{\mathfrak{g}},\overline{K})$-module, but if $(\mathfrak{g},K)$ is  equipped with a real structure $\sigma$, then the conjugate contragredient becomes an admissible Harish-Chandra module $V^{\langle\sigma\rangle}$   for  $(\mathfrak{g},{{K}})$, called the \emph{$\sigma$-twisted dual} of $V$ (compare \cite[Section 8]{AdamsEtAl2015}, where this is called the \emph{$\sigma$-Hermitian dual}). 
 
This construction is easily generalized to the case of families, at least  when the algebraic family of groups   $ {\boldsymbol{K}} $ is constant and reductive.  If $\mathcal F$ is a quasi-admissible algebraic family of Harish-Chandra modules for $(\boldsymbol{\mathfrak{g}},\boldsymbol{K})$, with isotypical decomposition 
\[  
\mathcal F = \oplus _{\tau\in \widehat K}   \mathcal{F}_\tau
\]
then we define its contragredient to be 
\[
\mathcal{F}^{\dagger} = \bigoplus _{ {\tau\in \widehat K}} \mathcal{F}^{\dagger}_{\tau} = \bigoplus _{ {\tau\in \widehat K}}\mathcal{H}\mathrm{om}_{O_X}(\mathcal{F}_{\tau},O_X) ,
\]
which is  a quasi-admissible family of Harish-Chandra modules for 
$(\boldsymbol{\mathfrak{g}},{\boldsymbol{K}})$. 

The complex conjugate of $ {\mathcal{F}^\dagger}$ is a quasi-admissible algebraic family of Harish-Chandra modules for the family $(\overline{\boldsymbol{\mathfrak{g}}}, \overline{\boldsymbol{K}})$ over $\overline X$.

\begin{defn}
Let $(\boldsymbol{\mathfrak{g}},{\boldsymbol{K}})$  be an algebraic family of Harish-Chandra pairs over $X$ for which $\boldsymbol{K}$ is a constant and reductive algebraic family of groups.   Let  $ \{\sigma_X,\sigma_{\boldsymbol{\mathfrak{g}}},\sigma_{\boldsymbol{K}}\}$   be a real structure on $(\boldsymbol{\mathfrak{g}},{\boldsymbol{K}})$. If  $\mathcal{F}$ is a quasi-admissible algebraic family of Harish-Chandra modules for  $(\boldsymbol{\mathfrak{g}},{\boldsymbol{K}})$, then  the \emph{$\sigma$-twisted dual of $\mathcal{F}$} is the sheaf 
\[
\mathcal{F}^{\langle\sigma\rangle} = \sigma_X^*\overline{\mathcal{F}^{\dagger}}
\]
 equipped with  the $(\boldsymbol{\mathfrak{g}},{\boldsymbol{K}})$-module structure obtained by composition with  the morphism of families of Harish-Chandra pairs 
\[
\xymatrix@C=40pt{
(\boldsymbol{\mathfrak{g}}, \boldsymbol{K}) 
\ar[r]^-{(\sigma_\mathfrak{g}, \sigma_{K})} &
 (\sigma_X^*\overline{\boldsymbol{\mathfrak{g}}}, \sigma_X^*{\overline{\boldsymbol{K}}})
 }
 \]
over $X$.
 \end{defn}

\section{The SU(1,1) family } 
\label{sec-su(1,1)-family}
In  \cite{Ber2016}  algebraic families over the complex projective line were attached to many  classical symmetric pairs of algebraic groups, and   real structures were constructed on these families. In this section we shall  remind the reader of the construction as it applies to $SL(2,\mathbb{C})$ and its diagonal subgroup, viewed as the fixed group of the involution
\[
\Theta\colon \begin{bmatrix} a & b \\ c & d \end{bmatrix} \longmapsto \begin{bmatrix} a & -b \\ -c & d \end{bmatrix} .
\]
In  \cite{Ber2016} we obtained a family over the projective line; here we shall consider only its restriction to the affine line,   which is enough for our purposes.

\subsection{An algebraic family of groups}

The family of complex algebraic groups is a subfamily of the constant family over $\C$ with fiber $ SL(2,\C){\times }SL(2,\C)$. 
The  fiber over $z \in \C$ is       
\begin{eqnarray*}
&& \boldsymbol{G}\vert _{z}=\left\{ \Bigl (\begin{bmatrix}
a &   b\\
z c & d
\end{bmatrix} , \begin{bmatrix}
a & z b\\
  c & d
\end{bmatrix} \Bigr ) :  ad-z bc=1  \right\} .
\end{eqnarray*}
It contains as a subfamily the constant family with fiber
\begin{eqnarray*}
&&  {K} =\left\{ \Bigl  (  \begin{bmatrix}
a & 0\\
0 & a^{-1}
\end{bmatrix} , \begin{bmatrix}
a & 0\\
0 & a^{-1}
\end{bmatrix} \Bigr ): a\in \C^*  \right\} .
\end{eqnarray*}
 The fibers of $\boldsymbol{\mathfrak{g}}$, the  family  of complex Lie algebras  corresponding to $\boldsymbol{G}$,  are 
\begin{eqnarray*}
&&\boldsymbol{\mathfrak{g}}\vert _{z}
=    \left \{ \Bigl (\begin{bmatrix}
a &   b\\
z c & -a
\end{bmatrix} , \begin{bmatrix}
a & z b\\
  c & -a
\end{bmatrix}\Bigr )  : a,b,c\in \mathbb{C} \right \} .
\end{eqnarray*}
The family $(\boldsymbol{\mathfrak{g}},{\boldsymbol{K}})$ is an algebraic family of Harish-Chandra pairs over $\C$.

\subsection{Real structure}

A real structure on   the family $(\boldsymbol{\mathfrak{g}},{\boldsymbol{K}})$   above is given by the   involution 
\[
  \sigma_{X}  \colon z \longmapsto \overline{z}
\]
on the base space $\C$, and the real structure
\[
\sigma_{\boldsymbol{G}} 
\colon 
 \boldsymbol{G}\vert _z \longrightarrow   \overline{ \boldsymbol{G}\vert _{\overline{z}}} ,
\qquad 
 \sigma_{\boldsymbol{G}}  \colon  (S,T)  
  \longmapsto  
 (\Theta(T^{-1})^*, \Theta(S^{-1})^* ) 
 \]
on $\boldsymbol{G}$, which determines real structures   
on  $\boldsymbol{\mathfrak{g}}$ and $\boldsymbol{K}$.
The explicit formula for the real structure on $\boldsymbol{G}$ is 
 \[
 \sigma_{\boldsymbol{G}} 
 \colon
  \Bigl (
  \begin{bmatrix}
a &   b\\
z c & d
\end{bmatrix}
 , 
\begin{bmatrix}
a & z b\\
  c & d
\end{bmatrix}
\Bigr ) 
\longmapsto
 \Bigl (
 \begin{bmatrix}
\overline d &   \overline c \\
\overline{z b} & \overline a
\end{bmatrix}
 , 
 \begin{bmatrix}
\overline d & \overline{z c} \\
 \overline{b} &  \overline a
\end{bmatrix}
\Bigr )  ,
 \] 
from which it is clear that the fixed groups of the involution over $\R\subseteq \C^{\sigma}$ are the real groups 
 \[
 \boldsymbol{G}\vert _{x}^{\sigma}=\left\{ \left( 
\begin{bmatrix}
 a &   b\\
x \overline{b}  & \overline a
\end{bmatrix}
 , 
 \begin{bmatrix}
 a & x b\\
  \overline{b}  & \overline a
\end{bmatrix}
\right)  :  |a|^2 - x |b|^2 =1 \right\} .
\]
So we see that 
\[
\boldsymbol{G}\vert _{x}^{\sigma}\cong\begin {cases} SU(1,1) &  x >0\\
 U(1)\ltimes \C & x=0\\
 SU(2) & x <0 ,
 \end{cases}
 \]   
where the action of $U(1)$ on $\C$ is scalar multiplication by the squares of elements in $U(1)$.

\subsection{The Casimir section}
 \label{subsec-casimir}
Associated to the algebraic family of Lie algebras $\boldsymbol{\mathfrak{g}}$ there is a family of enveloping algebras.
It is generated by the sections 
\begin{eqnarray*}
&&H: z \longmapsto  \Bigl ( \begin{bmatrix}
1 &   0\\
0 & -1 
\end{bmatrix}
 ,   
 \begin{bmatrix}
1  & 0\\
 0  & -1
\end{bmatrix}
 \Bigr ) ,
 \\ 
 \nonumber
&&E\,: z\longmapsto  \Bigl ( 
\begin{bmatrix}
0 &   \phantom{..}1\phantom{.}\\
0  &\phantom{.}0
\end{bmatrix}
 ,  
 \begin{bmatrix}
0 & \phantom{..}z  \\
 0  & \phantom{..}0
\end{bmatrix}  
\Bigr )  ,
\\ 
\nonumber
&&F\,\,: z\longmapsto  \Bigl ( 
\begin{bmatrix}
0 &   \phantom{..}0\\
z & \phantom{..}0 
\end{bmatrix}  
,
\begin{bmatrix}
0 & \phantom{..}0\\
 1  & \phantom{..}0
\end{bmatrix} 
\Bigr ) .
\end{eqnarray*}
The sections   $H,E,F$ are almost a standard $\mathfrak{sl}(2)$-triplet: 
\begin{eqnarray} \label{87}
&&[H,E ]=2E , \quad [H,F ]=-2F,\quad \text{and}\quad [E,F ]=zH .
\end{eqnarray}
Of special concern to us is the  \emph{Casimir section}
\begin{eqnarray*}
&&C=H^2 + \frac{2}{z}EF+ \frac{2}{z}FE =  H^2+2H+\frac{4}{z}FE
\end{eqnarray*} 
over $\C\setminus \{ 0\}$, which  is invariant under the adjoint action of $\boldsymbol{G}$.
(This is the same as the Casimir section considered in \cite{Ber2016}, but it is expressed here in terms of different generators.)

\subsection{Generically  irreducible families of modules}
\label{sec-twisted-duals}

In this section we     recall some facts about quasi-admissible and  generically irreducible families of Harish-Chandra modules for $(\boldsymbol{\mathfrak{g}},{\boldsymbol{K}})$ that were noted in \cite{Ber2016}. 

Since we are working over the base space $X=\C$, we can and will represent sheaves of $O_X$-modules  by their global sections, which are modules over the ring $\mathcal O$ of polynomial functions on $\C$. 
Let $\mathcal K$ be the field of rational functions on $\C$.  
A quasi-admissible family $\mathcal{F}$ of Harish-Chandra modules for $(\boldsymbol{\mathfrak{g}},{\boldsymbol{K}})$  is \emph{generically irreducible} if the  representation of the Lie algebra  $\mathcal K \otimes _{\mathcal O} \boldsymbol{\mathfrak {g}}$ on $\mathcal{K} \otimes _{\mathcal O} \mathcal {F}$ is   irreducible over  the algebraic closure of $\mathcal K$.  
Equivalently, in the present context, $\mathcal F$ is generically irreducible if the fiber $\mathcal F \vert _z$ is an irreducible $\boldsymbol{\mathfrak{g}}\vert _z$-module for all except at most countably many $z$.

If $\mathcal F$ is quasi-admissible and generically irreducible, then the  
    Casimir section $C$ acts as multiplication by a regular function $c_{\mathcal F}$  on $\C\setminus \{ 0\}$.   This is a first invariant of quasi-admissible and generically irreducible families of Harish-Chandra modules.
A second invariant is the set of $ {K}$-types, or weights, that occur in   a quasi-admissible and generically irreducible  family.  This set must  agree with the set of weights of some irreducible $(\mathfrak{sl}(2,\mathbb{C}),K)$-module, where $K$ is the diagonal subgroup in $SL(2,\C)$.  All the nonzero weights have multiplicity one. 

These two invariants fall quite far short of determining  quasi-admissible and generically irreducible  families up to isomorphism, but they do determine the families up to rational isomorphism: 

  \begin{proposition}
Let  $\mathcal {F}_1$ and $\mathcal{F}_2$ be two quasi-admissible and generically irreducible families of Harish-Chandra modules for  $(\boldsymbol{\mathfrak{g}},{\boldsymbol{K}})$.  If they have the same Casimir and weight invariants, then 
there is an isomorphism of $\mathcal K$-vector spaces
\[
\mathcal{K} \otimes _{\mathcal O} \mathcal {F}_1 \stackrel \cong \longrightarrow 
\mathcal{K} \otimes _{\mathcal O} \mathcal {F}_2
\]
that intertwines the actions of $\mathcal{K}\otimes _{\mathcal O} \boldsymbol{\mathfrak{g}}$. \end{proposition}

\begin{proof}
It is clear from the formulas \eqref{87} that 
 \[
 \mathcal{K}\otimes _{\mathcal O} \boldsymbol{\mathfrak{g}} \cong \mathfrak{sl}(2, \mathcal K).
 \]
Now from a generically irreducible quasi-admissible family of Harish-Chandra modules   for  $(\boldsymbol{\mathfrak{g}},\boldsymbol{K})$ we obtain  an    Harish-Chandra module for  $(\mathfrak{sl}(2, \mathcal K), K_{\mathcal K})$  (the underlying vector space of the module is a $\mathcal K$-vector space) such that 
\begin{enumerate}[\rm (i)]
\item the representations space  decomposes into integer weight spaces for the action of $  K_{\mathcal K}$, 
\item the Casimir for $\mathfrak{sl}(2,\mathcal K)$ acts as multiplication by an element of $\mathcal K$, and 
\item the representation is irreducible over the algebraic closure of $\mathcal K$.  
\end{enumerate}
The classification of such modules, up  to equivalence, is carried out exactly as in the standard case of complex representations of $(\mathfrak{sl}(2, \C), K_{\C})$ (and incidentally,  in the presence of (i) and (ii), irreducibility over $\mathcal K$ is equivalent to irreducibility over the algebraic closure). The result follows.
\end{proof}

\subsection{The sigma-twisted duals}
The theorem from the previous section allows us to characterize those generically irreducible and quasi-admissible families that are rationally isomorphic to their $\sigma$-twisted duals:

\begin{proposition}
Let  $\mathcal{F}$  be a quasi-admissible and generically irreducible family of Harish-Chandra modules for  $(\boldsymbol{\mathfrak{g}},{\boldsymbol{K}})$.  There is an isomorphism of  $\mathcal{K}\otimes _{\mathcal O} \boldsymbol{\mathfrak{g}}$-modules
\[
\mathcal{K} \otimes _{\mathcal O} \mathcal {F} \stackrel \cong \longrightarrow 
\mathcal{K} \otimes _{\mathcal O} \mathcal {F}^{\langle\sigma\rangle}
\]
if and only if the regular Casimir  function $c_{\mathcal F}$ on $\C\setminus \{ 0 \}$ is real-valued on $\R \setminus \{0\}$. 
\end{proposition}

\begin{proof}
This follows from the proposition in the previous section, together with  the fact that the weights of $\mathcal {F}$ and $\mathcal {F}^{\langle\sigma\rangle}$ are equal, while 
\[
 c_{\mathcal {F}^{\langle\sigma\rangle}} = \sigma(c_{\mathcal F}),
\]
where $\sigma(c)(z) = \overline{c(\overline z)}$.
\end{proof}

\section{Families of unitary representations}
\label{sec-group-reps}

In this section we shall indicate how one can use  Jantzen filtration techniques to obtain families of unitary representations from quasi-admissible and generically irreducible families of Harish-Chandra modules.  The families will include both discretely- and continuously-parametrized parts.  In the next section we shall show that the families of unitary representations that may be obtained in this way include the contraction family that we described in the introduction.

\subsection{The Jantzen filtration}

Let $\mathcal{F} $ and $\mathcal {H}$ be quasi-admissible algebraic families of Harish-Chandra modules for $(\boldsymbol{\mathfrak{g}},\boldsymbol{K})$ and let 
\begin{equation}
\label{eq-simple-map}
\varphi \colon  \mathcal{F}  \longrightarrow  \mathcal{H}   
\end{equation}
be a morphism that is generically an isomorphism.  In this section we shall review the construction of canonical increasing and decreasing filtrations of the fibers of $\mathcal F $ and $\mathcal H$, respectively (the Jantzen filtrations; compare   \cite{Jantzen79,Humphreys08}).  We shall also recall that  a choice of   coordinate near any point determines isomorphisms between corresponding subquotients at that point. 

We shall continue to assume that  $\boldsymbol{K}$ is a constant family of reductive groups with fiber $K$, and we shall continue to think  of $\mathcal{F}$ and $\mathcal{H}$ as free  modules over the  algebra of regular functions on the line,  with compatible actions of $K$ and   the Lie algebra of global sections of $\boldsymbol{\mathfrak{g}}$.  
To be precise, rather than \eqref{eq-simple-map}  we shall start with a  map
\begin{equation}
\label{eq-less-simple-map}
\varphi \colon\mathcal{K}\otimes_{\mathcal{O}} \mathcal{F}  \longrightarrow \mathcal{K}\otimes_{\mathcal{O}}\mathcal{H}   
\end{equation}
that is a  $(\boldsymbol{\mathfrak{g}}, {K})$-equivariant isomorphism of $\mathcal{K}$-vector spaces, 
where  $\mathcal{O}$ is the algebra of regular functions on the line 
and  $\mathcal{K}$ is the field of rational functions.

Let $z$ be a point on the affine line, and denote by  $\mathcal{O}_z$ the localization of $\mathcal{O}$ at $z$,  comprised of those rational functions  whose denominators are nonzero at $z$. In addition form  the localizations
\[
\mathcal{F}_z = \mathcal{O}_z\otimes_{\mathcal{O}} \mathcal{F}  
\quad\text{and} \quad 
\mathcal{H}_z = \mathcal{O}_z\otimes_{\mathcal{O}} \mathcal{H}.
\]
We shall think of these as  $ \mathcal{O}_z$-submodules of  $\mathcal{K}\otimes_{\mathcal{O}} \mathcal{F} $ and $ \mathcal{K}\otimes_{\mathcal{O}}\mathcal{H}$, respectively.
 Fix a coordinate $p$ at $z$,  or in other words a degree-one polynomial function that vanishes at $z$, and for $n\in \mathbb Z$ define 
\[
\mathcal{F}_z^n = \{\,  f \in \mathcal{F}_z : \varphi(f) \in p^n \mathcal{H}_z\, \}
\]
and 
\[
\mathcal{H}_z^n = \{\,  h \in \mathcal{H}_z : p^n h \in\varphi[\mathcal{F}_z^n]\, \} .
\]
These constitute decreasing and increasing filtrations of $\mathcal F_z$ and $\mathcal H_z$, respectively, by $\mathcal{O}_z$-submodules. 

Consider the fibers of $\mathcal F$ and $\mathcal H$,
\[
\mathcal{F}\vert_z =  \C \otimes _\mathcal{O} \mathcal {F} = \C \otimes _{\mathcal{O}_z} \mathcal {F}_z
\quad \text{and} \quad 
\mathcal{H}\vert_z =  \C \otimes _\mathcal{O} \mathcal {H} = \C \otimes _{\mathcal{O}_z} \mathcal {H}_z ,
\]
which are complex vector spaces and $(\boldsymbol{\mathfrak{g}}\vert_z, K)$-modules.
Denote by 
\[
\mathcal{F}\vert_z^n\subseteq \mathcal{F}\vert_z
\] the image in the complex vector space $\mathcal{F}\vert_z$  of the morphism  
\[
 \C \otimes _\mathcal{O} \mathcal {F}_z^n  \longrightarrow  \C \otimes _\mathcal{O} \mathcal {F}_z , 
 \]
 and similarly  denote  by 
 \[
\mathcal{H}\vert_z^n\subseteq \mathcal{H}\vert_z
\] the image in $\mathcal{H}\vert_z$  of the morphism of complex vector spaces
\[
 \C \otimes _\mathcal{O} \mathcal {H}_z^n  \longrightarrow  \C \otimes _\mathcal{O} \mathcal {H}_z . 
 \]
 These are  decreasing and increasing filtrations of $\mathcal{F}\vert_z$ and $\mathcal{H}\vert_z$, respectively, by $(\boldsymbol{\mathfrak{g}}\vert_z, K)$-submodules.
 
We shall now obtain isomorphisms between  subquotients of these two filtrations.  The formula 
 \[
 f \longmapsto p^{-n}\varphi(f)
 \]
  defines an isomorphism of $\mathcal{O}_z$-modules 
from $
 \mathcal{F}_z^n$ to $ \mathcal{H}_z^n$ and induces a surjective morphism of $(\boldsymbol{\mathfrak{g}}\vert_z, K)$-modules
  \[
  \varphi^n \colon  \mathcal{F}\vert _z^n\longrightarrow \mathcal{H}\vert_z^n .
 \]

 \begin{proposition}
 \label{prop-jantzen-subquotients}
The above morphism $\varphi^n$ maps the subspace  $
 \mathcal{F}_z^{n+1}\subseteq  
 \mathcal{F}_z^n$ into the subspace $
 \mathcal{H}_z^{n-1}\subseteq  
 \mathcal{H}_z^n$
 and induces an isomorphism of $(\boldsymbol{\mathfrak{g}}\vert_z, K)$-modules
 \[
  \varphi^n \colon  \mathcal{F}\vert _z^n\big /\mathcal{F}\vert _z^{n+1}\stackrel \cong \longrightarrow \mathcal{H}\vert_z^n \big /\mathcal{H}\vert _z^{n-1}
 \]
 \textup{(}that depends  on the differential  of $p$ at $z$\textup{)}.
  \end{proposition}
  
  The proof is a simple calculation, but for the convenience of the reader we shall give it  in a moment.  The proof shows that there is an isomorphism of subquotients, as above, whether or not the morphism $\varphi$ in \eqref{eq-less-simple-map} is assumed to be an isomorphism.   But the hypothesis that $\varphi$ is an isomorphism implies the following additional important fact:
  
  \begin{proposition}
   \label{prop-jantzen-subquotients2}
  If the morphism $\varphi \colon\mathcal{K}\otimes_{\mathcal{O}} \mathcal{F}  \to \mathcal{K}\otimes_{\mathcal{O}}\mathcal{H}$ is an isomorphism, then the filtrations $ \mathcal{F}\vert_z^\bullet $ and $ \mathcal{H}\vert_z^\bullet $ 
 are  {exhaustive}:  the intersections of the filtration spaces are   zero, while the unions are  $\mathcal{F}\vert_z$ and $\mathcal{H}\vert _z$, respectively.
  \end{proposition}
  
  \begin{proof}
The modules $ \mathcal{F}_z$ and $ \mathcal{H}_z$ decompose  into  direct sums of free and finite rank submodules according to the $K$-isotypical decompositions of $\mathcal{F}$, and $ \mathcal{H}$ and the morphism $\varphi$ respects these decompositions.  So it suffices to prove the proposition for a single summand, where it is easy.
  \end{proof}
  
   \begin{proof}[Proof of the first proposition]
    If $f\in \mathcal{F}_z^n$, and if $f$ determines an element of the subspace $\mathcal{F}\vert_z^{n+1}\subseteq  \mathcal{F}\vert _z^n
 $,
 then we can decompose $f$ as 
 \[
 f =  f_1 +  p f_2 ,
 \]
 where $f_1 \in \mathcal{F}_z^{n+1} $ and $f_2\in \mathcal{F}_z$.  It follows from the formula
 \[
 f_2 = p^{-1} (f - f_1) 
 \]
that in fact $f_2\in \mathcal{F}^{n-1}_z$.  Consider now  the formula 
\[
p^{-n}\varphi(f) = p\cdot p^{-(n+1)} \varphi(f_1) + p^{-(n-1)}\varphi(f_2) .
\]
 The first term on the right lies in $p \cdot \mathcal{H}_z$ while the second term on the right  lies in $\mathcal{H}^{n-1}_z$.  So the image of the class of $f$ in $\mathcal{F}\vert_z^n$ under the morphism $\varphi^n$ lies in $ \mathcal{H}\vert_z^{n-1} $, as required.

Surjectivity of the induced map on quotient spaces  is clear.  As for injectivity, if an element of $\mathcal{F}\vert_x^n/ \mathcal{F}\vert_z^{n+1}$ is  represented by some $f\in \mathcal{F}_z^n$ and  maps under $\varphi^n$ to an element of $\mathcal{H}\vert_z^{n-1}$, then we may write 
 \[
 p^{-n}\varphi(f) = p^{-(n-1)} \varphi (f_1) + p h_1
 \]
 for some $f_1\in \mathcal{F}_z^{n-1}$ and some $h_1\in \mathcal{H}_z$.   But then
 \[
 \varphi(f - pf_1) =  p^{n+1} h_1 ,
 \]
 which implies that $f-pf_1 \in \mathcal{F}_z^{n+1}$, so that $f$ determines the zero element of $\mathcal{F}\vert_z^n/ \mathcal{F}\vert_z^{n+1}$, as required.
 \end{proof}

\subsection{The Jantzen filtration and real structures}

Let   $(\boldsymbol{\mathfrak{g}},\boldsymbol{K})$ be an algebraic family of  Harish-Chandra pairs over the complex affine line $\C$, as in the previous section,  with $\boldsymbol{K}$   a constant family of connected reductive groups with fiber $K$.   Assume that   $(\boldsymbol{\mathfrak{g}},\boldsymbol{K})$ is equipped with a real structure compatible with the standard real structure $z\mapsto \overline z $ on $\C$.   Let $\mathcal F$ be a quasi-admissible and generically irreducible algebraic family of Harish-Chandra modules for $(\boldsymbol{\mathfrak{g}},\boldsymbol{K})$, and  suppose that there is an isomorphism 
\[
\varphi \colon\mathcal{K}\otimes_{\mathcal{O}} \mathcal{F}  \longrightarrow \mathcal{K}\otimes_{\mathcal{O}}\mathcal{F}^{\langle\sigma \rangle}  .
\]
In this section we shall  show that the Jantzen subquotients $\mathcal{F}\vert _x ^n / \mathcal{F}\vert_x ^{n+1}$  may be equipped with canonical (up to real scalar factors)  nondegenerate hermitian forms that are $\sigma$-invariant for the actions of the   Harish-Chandra pairs   $(\boldsymbol{\mathfrak{g}} \vert _x , K ) $.

We shall identify  $\mathcal{F}^{\langle\sigma \rangle}$ with the space of functions from $\mathcal F$ to $\mathcal O$ that are conjugate $\mathcal O$-linear in the sense  that if $e\in \mathcal{F}^{\langle \sigma\rangle}$, $q\in \mathcal{O}$ and   $f \in \mathcal{F}$, then 
\[
e(qf) = \sigma(q)e(f) ,
\]
 where $\sigma(q)(z) = \overline{q(\overline{z})}$, and that vanish on all but finitely many of the $K$-isotypical summands of $\mathcal{F}$.  The $\mathcal{O}$-module structure is $(q\cdot e )(f)  = q \cdot e(f)$ and the $\boldsymbol{\mathfrak{g}}$-module structure is $(X\cdot e)(f) = - e(\sigma (X)f)$.
We obtain from the isomorphism $\varphi$  above  a  complex-sesquilinear map 
\[
\langle\,\,\,,\,\,\rangle\colon \left ( \mathcal{K}\otimes_{\mathcal{O}} \mathcal{F} \right ) \times \left (  \mathcal{K}\otimes_{\mathcal{O}} \mathcal{F} \right )  \longrightarrow \mathcal{K}
\]
using the formula 
$ \langle f_1,f_2\rangle = \varphi(f_1)(f_2)
$.
This map is in fact $\mathcal{K}$-sesquilinear in the sense that 
\[
\langle q_1f_1,q_2 f_2\rangle =  q_1   \langle f_1,f_2\rangle  \sigma(q_2),
\]
for all $f_1,f_2,\in \mathcal{K}\otimes_{\mathcal{O}} \mathcal{F} $ and all $q_1,q_2 \in \mathcal{K}$, where again $\sigma(q)(z) = \overline{q(\overline{z})}$. The pairing  is also $\sigma$-invariant  under the  $ \boldsymbol{\mathfrak {g}}$-action and the $\boldsymbol{K}$-action in the sense that
\[
\langle X  \cdot f _1 , f_2 \rangle  + \langle f_1,  \sigma(X)\cdot f_2 \rangle = 0 
\]
\[
\langle g  \cdot f _1 , \sigma_{\boldsymbol{K}} (g)  \cdot f_2 \rangle  =\langle f_1,  f_2 \rangle  
\]
for all $f_1,f_2,\in \mathcal{K}\otimes_{\mathcal{O}} \mathcal{F} $,  all $X\in \boldsymbol{\mathfrak {g}}$, and all $g$ that are global sections of $\boldsymbol{K}$.

By Schur's lemma, any other isomorphism 
\[
\psi \colon\mathcal{K}\otimes_{\mathcal{O}} \mathcal{F}  \longrightarrow \mathcal{K}\otimes_{\mathcal{O}}\mathcal{F}^{\langle\sigma \rangle}   
\]
is equal to $\varphi$ times a rational function $q\in \mathcal{K}$.  Note that the sesquilinear form associated to $\psi$ is 
$ q \cdot  \langle f_1,f_2\rangle$.
Returning to $\varphi$, it follows from this  that there is some $q \in \mathcal{K}$ such that 
\[
\sigma( \langle f_2,f_1\rangle ) = q \cdot \langle  f_1,  f_2\rangle
\]
for all $f_1,f_2,\in \mathcal{K}\otimes_{\mathcal{O}} \mathcal{F} $. This is  because the left-hand side is the $\mathcal{K}$-sesquilinear form  associated to the isomorphism 
\[
\psi\colon  \mathcal{K}\otimes_{\mathcal{O}} \mathcal{F}  \longrightarrow \mathcal{K}\otimes_{\mathcal{O}}\mathcal{F}^{\langle\sigma \rangle}  
\]
defined by 
$
\psi(f_1)(f_2) = \sigma( \langle f_2,f_1\rangle ) $.
Note that since $\sigma$ is an involution, 
\[
\sigma(q) \cdot q = 1 .
\]
Hence  the rational function $q$ has no zeros or poles on the real axis.

Now fix  $x\in \R$ and choose a real  coordinate $p$ at $x$ (meaning that   $\sigma(p) = p$).   The Jantzen construction of the previous section produces   nondegenerate, invariant, complex-sesquilinear forms
\[
\langle\,\,\,,\,\,\rangle_{x,n}\colon 
\mathcal{F}\vert_x^n/ \mathcal{F}\vert_x^{n+1} \times \mathcal{F}\vert_x^n/ \mathcal{F}\vert_x^{n+1} \longrightarrow \C .
\]
The space $\mathcal{F}\vert_x^n$ is determined by elements $f_1\in \mathcal{F}_x$ such that 
\[
\operatorname{ord}_x ( \langle f_1,f_2\rangle ) \ge n
\]
for every $f_2 \in \mathcal{F}_x$, and  the scalar  $ \langle f_1,f_2\rangle_{x,n}$ is equal to the value  of the function $ p^{-n} \langle f_1,f_2\rangle$ at $x$. 

Up to a shift in the index $n$, the quotient spaces $\mathcal{F}\vert_x^n/ \mathcal{F}\vert_x^{n+1}$,  are  independent of $\varphi$ and also independent of the choice of $p$.  Moreover up to a real scalar factor the above forms are also independent of the choice of  $p$.  Finally since 
\[
\overline{\langle f_2 , f_1 \rangle}_{x,n}  =  q(x) \cdot   { \langle f_1 , f_2 \rangle_{x,n}} ,
\]
where $|q(x)| = 1$, after rescaling the above forms by a square root of $q(x)$ we obtain, for all $n$,    Hermitian forms that are independent of all choices, up to real scalar factors.

So for each $x\in \R$ we have constructed    a canonical (finite) set of   $(\boldsymbol{\mathfrak{g}}\vert_x, K)$-modules (the Jantzen quotients).    Each is equipped with a canonical, up to a real scalar factor, nondegenerate Hermitian form.   All the forms are   $\sigma$-invariant for the action of $(\boldsymbol{\mathfrak{g}}\vert_x, K)$. 

\subsection{Distinguished parameter values}
\label{subsec-distinguished}

For all but a countable set of values $x$, the Jantzen filtration on $\mathcal{F}\vert_x$ will be trivial, so that there will be a unique Jantzen quotient at $x$, namely $\mathcal{F}\vert _x$ itself.  As a result the filtration immediately selects   a distinguished, discrete collection of parameter values, where there is   more than one Janzten quotient, and a corresponding discretely parametrized set of representations.  

Obviously if the fiber is an irreducible module, then the filtration must be trivial.  In the particular special case that  we shall analyze in the next section the distinguished values are precisely the reducibility points, but this need not be so in general.

\subsection{Infinitesimally unitary Jantzen quotients}

In this section we shall continue with the notation of the previous section; in particular we shall denote by $x$ a fixed element of  the real line. We shall consider the problem of integrating the Lie algebra representation of $\boldsymbol{\mathfrak {g}}\vert ^\sigma_x$ on a Jantzen quotient  $\mathcal{F}\vert_x^n/ \mathcal{F}\vert_x^{n+1}$ so as to obtain a unitary representation.  The arguments in this section no longer involve families; they just involve the hermitian forms constructed above.

\begin{defn}
We shall say that a Jantzen quotient $\mathcal{F}\vert_x^n/ \mathcal{F}\vert_x^{n+1}$ is \emph{infinitesimally unitary} if the    real one-dimensional space of   hermitian forms on $\mathcal{F}\vert_x^n/ \mathcal{F}\vert_x^{n+1}$ that was constructed in the previous section includes a   positive-definite hermitian form (that is, an inner product).
\end{defn}

Assume that $\mathcal{F}\vert_x^n/ \mathcal{F}\vert_x^{n+1}$ is infinitesimally unitary. Adjust the sesquilinear form $\langle\,\,\,,\,\,\rangle_{x,n}$ by a (uniquely determined) complex scalar factor, if necessary, so as to make it  an inner product.  
We can apply  the following theorem of Nelson  \cite[Theorem 5]{Nelson59}  to show that   the admissible $(\boldsymbol{\mathfrak {g}}\vert^\sigma _x, K^\sigma)$-module   $\mathcal{F}\vert_x^n/ \mathcal{F}\vert_x^{n+1}$ integrates to a unitary representation on the Hilbert space completion of the Jantzen quotient.
 
 \begin{theorem}
 Let $\mathfrak{g}$ be a real, finite-dimensional Lie algebra of skew-symmetric operators acting on a common invariant and dense domain $\mathcal F$ in a Hilbert space $H$.  Let $X_1,\dots, X_k$ be a basis for $\mathfrak{g}$ and form the symmetric operator
 \[
 \Delta = X_1 ^2 + \cdots + X_k^2 
 \]
 with domain $\mathcal F$.  If $\Delta$ is essentially self-adjoint \textup{(}that is, if $\Delta$ has a unique self-adjoint extension\textup{)}, then each $X_j$ is essentially skew-adjoint, and if $G$ is the simply connected Lie group associated to $\mathfrak {g}$, then there is a unique unitary representation 
 \[
 \pi\colon G \longrightarrow U(H)
 \]
for which the generators of the one-parameter unitary groups $\{ \pi(\exp(tX_j))\} _{t\in \R}$ are the skew-adjoint extensions of the operators $X_j$.
 \end{theorem}
 
 In order to apply Nelson's theorem we shall assume that there is an internal  vector space direct sum decomposition 
 \[
  \boldsymbol{\mathfrak{g}}\vert _{x}^{\sigma} =   {\mathfrak{k}}^{\sigma} \oplus   \mathfrak{p}_{x} ^{\sigma}
  \]
  in which $ \mathfrak{p}_{x}^{\sigma} $ is a real subspace of $  \boldsymbol{\mathfrak{g}}\vert _{x}^{\sigma}$ that is invariant under the action of $K^\sigma$.  We shall assume that both $  {\mathfrak{k}}^{\sigma} $ and $ {\mathfrak{p}}_x ^{\sigma}$ carry $K^\sigma$-invariant inner products.  This is certainly the case in our   example. Indeed it is the case whenever the real group $K^\sigma$ is compact.

Fix orthonormal bases of  $  {\mathfrak{k}}^{\sigma} $ and $ {\mathfrak{p}}_x^{\sigma} $ with respect to the $K^\sigma$-invariant inner products, and combine them to form a basis $\{ X_j\}$ for  $\boldsymbol{\mathfrak{g}}\vert_{x}^{\sigma}$. The operator 
\[
\Delta  \colon \mathcal{F}\vert_x^n/ \mathcal{F}\vert_x^{n+1}
\longrightarrow 
\mathcal{F}\vert_x^n/ \mathcal{F}\vert_x^{n+1}
\]
in Nelson's theorem is symmetric.  It is also $K^\sigma$-invariant, and so it   leaves invariant  the $K^\sigma$-isotypical summands in $\mathcal{F}\vert_x^n/ \mathcal{F}\vert_x^{n+1}$, which are finite-dimen\-sional.  This implies that $\Delta$ is essentially self-adjoint, as required.

We conclude that the representation of $ \boldsymbol{\mathfrak{g}}\vert_{x}^{\sigma}$ on each infinitesimally unitary Jantzen quotient integrates to a unitary representation of the universal cover of $ \boldsymbol{{G}}\vert _{x}^{\sigma}$ on the Hilbert space 
 completion of the Jantzen quotient.
 
By definition of a $(\boldsymbol{\mathfrak {g}}\vert^\sigma _x, K^\sigma)$-module, we also know that the action of $\mathfrak{k}^\sigma\subseteq \boldsymbol{\mathfrak{g}}\vert_{x}^{\sigma}$ integrates to a unitary action of $K^\sigma$.  So if the inclusion of $K^\sigma$ into $\boldsymbol{  {G}}^\sigma\vert _x$ induces a surjection 
\[
\pi_1 (K^\sigma) \longrightarrow \pi_1 (\boldsymbol{  {G}}\vert _x^\sigma) ,
 \]
 then in fact the above unitary representation of the universal covering group of $ \boldsymbol{{G}}\vert_{x}^{\sigma}$ factors through $ \boldsymbol{{G}}\vert_{x}^{\sigma}$.  This is the case in our example (although it is not the case in all the examples in \cite{Ber2016} constructed from symmetric pairs of reductive groups).

 \section{Contraction families}

In this final section we shall determine the infinitesimally  unitary Jantzen quotients associated to a  family of quasi-admissible and generically irreducible modules for  the   family of Harish-Chandra paris $(\boldsymbol{\mathfrak{g}}, \boldsymbol{K})$ described in Section~\ref{sec-su(1,1)-family}.   We shall recover in this way the contraction family of unitary representations described in the introduction.  

The infinitesimally unitary Jantzen quotients are shown in Figure~\ref{fig:contraction-family}.  The quotients include both a continuous family of unitary principal series representations for $SU(1,1)$ and  a discrete sequence of (finite-dimensional) irreducible representations for $SU(2)$. The figure  gives some sense of the convergence phenomenon that is of interest in the theory of contractions.

 \begin{figure}[ht] 
  \centering
  \includegraphics[scale=0.20]{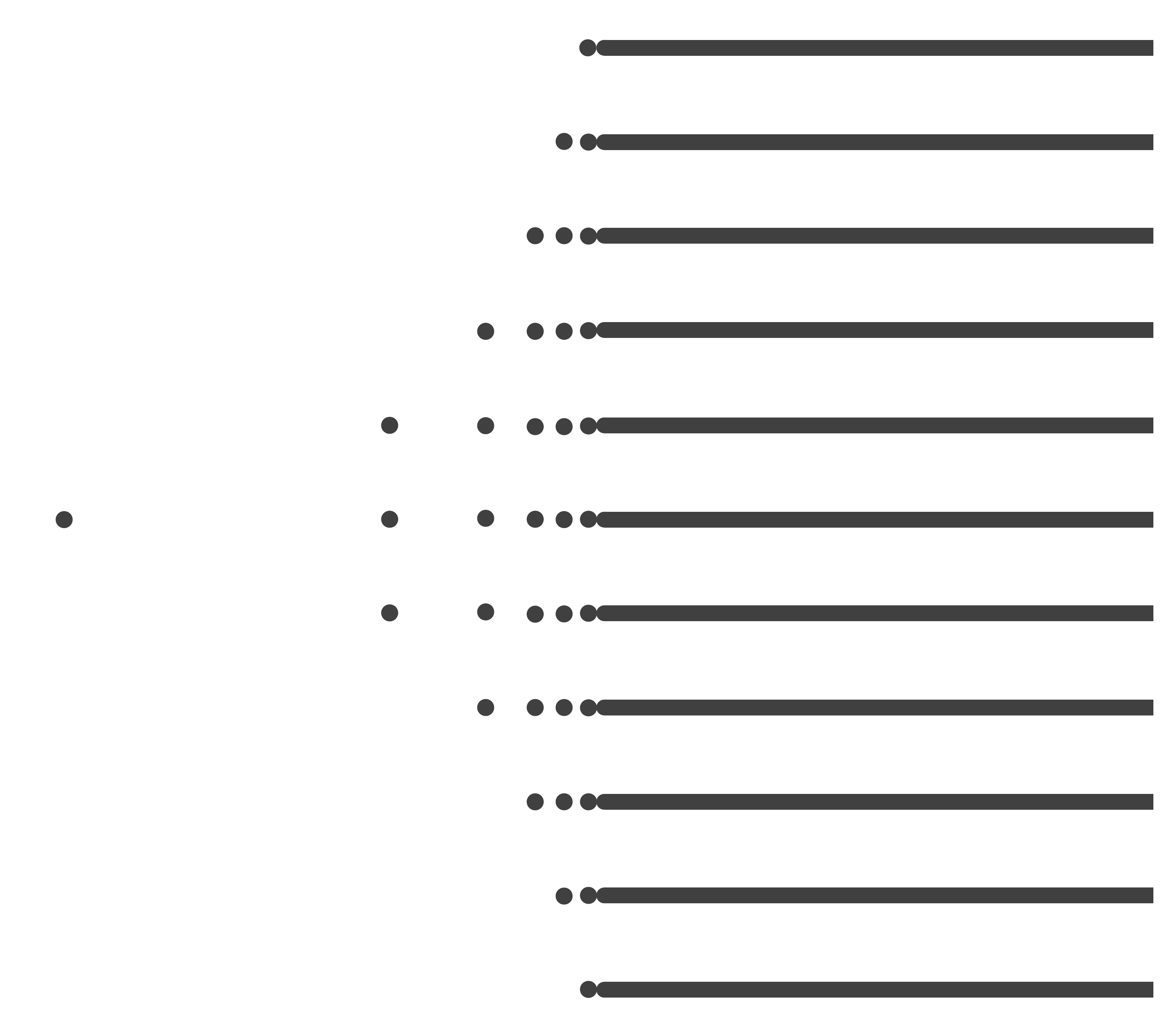} 
  \captionsetup{width=0.8\textwidth}
  \caption{\small  The diagram illustrates the  infinitesimally unitary Jantzen quotients for the unique  family with Casimir $ - ({1+x})/{x}$ that is generated by its $0$-weight space.  The real parameter $x$ runs horizontally.  At every $x$, the weights that appear in a unitary Jantzen quotient of $\mathcal F\vert _x$  are shown vertically.
   For $x<0$ there is a unitary quotient  for a sequence of values   $x = -(2m+1)^{-2}$ converging to zero, giving    finite-dimensional unitary representations of   $SU(2)$. At  $x=0$  there is an irreducible unitary representation of the Cartan motion group that includes every even weight.  When 
  $x>0$ every even weight occurs once agian, and the representations are the unitary spherical principal series for $SU(1,1)$ with Casimir converging to infinity as $x$ converges to zero. }
  \label{fig:contraction-family}
\end{figure}

\subsection{Bases for  generically irreducible families}\label{canonical bases}

Let  $(\boldsymbol{\mathfrak{g}},{\boldsymbol{K}})$ be the algebraic family of Harish-Chandra pairs from Section~\ref{sec-su(1,1)-family}.  For simplicity, from here on we shall be calculating with those families    that have  non\-zero $K$-isotyopical  components precisely for every even weight (recall that $K\cong \C^{\times}$).   These are the families that are relevant to the contraction families of unitary representations mentioned in the introduction; however the other possible types of families are listed in \cite{Ber2016} and they can be analyzed in a similar way. 

Let $\mathcal F$ be such a family.  As a consequence of our assumption, the $\mathcal{O}$-module  $\mathcal{F}$ has a basis $\{ f_k\}$ indexed by the even integers $k$, with 
\begin{align*} 
Hf_k &=k f_{k} 
\\  
 Ef_k &=A_k f_{k+2}\\  
 Ff_{k+2}&=B_{k}  f_{k}
\end{align*}
where $A_k, B_{k}\in \mathcal O$.  Here $H$, $E$ and $F$ are the sections of $\boldsymbol{\mathfrak{g}}$ defined  in Section~\ref{subsec-casimir}.  The Casimir from Section~\ref{subsec-casimir} acts on $f_n$ as multiplication by the polynomial 
\[
k^2+2k+\frac{4}{z}A_k B_k  
\]
(so this expression is independent of $k$).

Let us now make a further assumption:   that $\mathcal F$ is generated as a family of Harish-Chandra modules by its $0$-isotypical component $\mathcal F_0$.\footnote{Other natural choices are possible, and are described in \cite[4.9]{Ber2016}, but for brevity we shall focus on just this one example here.} Then $A_k$ is nowhere zero for $k\ge 0$.  This is because the isotypical component $\mathcal F_{k+2}$ is generated as a $\mathcal O$-module by the element
\[
A_kA_{k-2} \cdots A_0 f_{k+2} .
\]
Similarly, $B_k$ is nowhere zero for $k<  0$.  Hence 
\[
A_k = \mathrm{constant}  , \quad \text{ when $k\ge 0$}
\]
and 
\[
B_k = \mathrm{constant}  , \quad \text{ when $k <  0$} ,
\]
 where the constants  might be distinct, but are all nonzero.  
 
 In fact after adjusting the basis elements $f_k$ by scalar factors we can, and will from now on,   assume that all the above constants are $1$.   When that is done, the value of the Casimir section on $f_k$  becomes
 \[
k^2+2k+\frac{4}{z} B_k   , \quad \text{ when $k\ge 0$}
\]
and 
 \[
k^2+2k+\frac{4}{z} A_k   , \quad \text{ when $k<  0$} .
\]
So we see that the remaining $A_k$ and $B_k$ are completely determined by the action of the Casimir section.  Hence:

\begin{proposition}
Let $\mathcal{F} $ be a  quasi-admissible and generically irreducible algebraic family  of Harish-Chandra modules for $(\boldsymbol{\mathfrak{g}}, \boldsymbol{K})$.  Assume that  the weight $k$ isotypical summands of $\mathcal {F}$ are nonzero precisely for the even integer weights. If $\mathcal{F}$ is generated by its  weight $k$ isotypical summand,
then $\mathcal {F}$ is determined up to isomorphism by the action of the Casimir section. \qed
\end{proposition}

Conversely, a family of the type described in the proposition may be constructed for which the action of the Casimir section is by any given regular function on $\C\setminus \{0\}$, other than one of the constants   $k^2 + 2k$,  that has at most a simple pole at $0$   (a family can be constructed for the excluded constants, too, but it is not generically irreducible). See  \cite{Ber2016} for this and for further information on classification of quasi-admissible and generically irreducible algebraic families of modules for $(\boldsymbol{\mathfrak{g}}, \boldsymbol{K})$.

\begin{remark}
Similar  families are considered in \cite[Section14]{AdamsEtAl2015}. The concern there is   with \emph{constant} families of Harish-Chandra pairs, but in applying  Jantzen filtration techniques to our families we shall certainly be following the lead of  \cite{AdamsEtAl2015}.   \end{remark}

\subsection{Computation of the Jantzen quotients}

We shall work in this section with a family $\mathcal F$ as in the previous proposition, and we shall assume in this section that  there is an isomorphism 
\[
\varphi \colon\mathcal{K}\otimes_{\mathcal{O}} \mathcal{F}  \longrightarrow \mathcal{K}\otimes_{\mathcal{O}}\mathcal{F}^{\langle\sigma \rangle}   .
\]
This is so precisely when the value of the Casimir section is fixed by the involution $\sigma$.

The  $\sigma$-twisted dual  $\mathcal{F}^{\sigma}$  has the same weights as $\mathcal{F}$, namely a one-dimen\-sional $k$-isotyopical summand for every even integer $k$.   If we choose  basis elements $e_k\in \mathcal{F}_k^{\langle\sigma\rangle}$ such that $e_k(f_k) = 1$, then 
\begin{eqnarray*}
&&He_k=ke_{k} 
\\ 
&&Ee_k=-\sigma(B_{k})e_{k+2}\\ 
&&Fe_k=-\sigma(A_{k-2}) e_{k-2} .
\end{eqnarray*}
The isomorphism $\varphi$  can be described by a sequence of formulas
\[
\varphi (f_k) = \varphi_k e_k ,
\]
where  $\varphi_k\in \mathcal{K}$. The rational functions $\varphi_k$ are not independent of one another, since compatibility with the action of $E$ and $F$ implies that 
\begin{eqnarray*}
&&A_{k-2} \varphi_{k} =-\sigma(B_{k-2})\varphi_{k-2} \\
&&B_k \varphi_k=-\sigma (A_k) \varphi_{k+2}
\end{eqnarray*}
for all even integers $k$.  These relations imply in turn that 
\[
\varphi_{k} =(-1)^{\frac k2}  \frac{B_{k-2}}{\sigma(A_{k-2})} \frac{B_{k-4}}{\sigma(A_{k-4})}  \cdots 
\frac{B_{0}}{\sigma(A_{0})}  \varphi_0 \quad \text{if $k>0$}
\]
and that 
\[
\varphi_{k} = (-1)^{\frac k2}   \frac{\sigma(A_{k})}{B_{k}} \frac{\sigma(A_{k+2})}{B_{k+2 }}  \cdots 
\frac{\sigma(A_{-2})}{B_{-2}}   \varphi_0 \quad \text{if $k< 0$} .
\]
If we are working, as we may,  with a basis  $\{ f_n\}$ for $\mathcal{F}$ for which 
$
A_k = 1 $ when $k\ge 0$ 
and 
$
B_k = 1$  when $k <  0$,
then the formulas simplify to
\[
\varphi_{k} =(-1)^{\frac k2}   {B_{k-2}}  {B_{k-4}}   \cdots 
 {B_{0}}   \varphi_0 \quad \text{if $k>0$}
\]
and that 
\[
\varphi_{k} = (-1)^{\frac k2}     A_{k}   A_{k+2}   \cdots 
 A_{-2}    \varphi_0 \quad \text{if $k< 0$} .
\]
Once again, each of the $A$- or $B$-polynomials appearing in these formulas is explicitly determined by the value of the Casimir section.

Let us now consider the specific example where the action of the Casimir is given by the regular function   
\[
c_{\mathcal F}(z)=   - \frac{1+z}{z}
\]
on $\C\setminus \{ 0\}$ (this is    choice of Casimir that will lead to the contraction families from mathematical physics that were mentioned in the introduction).   We want to determine all the infinitesimally unitary Jantzen quotients.

If we work with the isomorphism $\varphi$ for which $\varphi_0 =1$, then  we  get 
\[
B_k 
= - \frac{ 1}{4} \left (   1 +  x(1+k)^2   \right ) 
\]
when $k\ge 0$ and hence 
\[
\varphi_{k }(x)  = \frac{1 + (k{-}1)^2 x }{4}\cdot   \frac{1 + (k{-}3)^2 x }{4} \cdot \cdots \cdot  \frac{1+ 9x }{4} \cdot \frac{1+x}{4} .
\]
The same formula holds when    $k<0$.  So unless $x =  -(2 m+1)^2$, for some $m$, all     the scalars $\varphi_k(x)$ are nonzero, and hence the Jantzen filtration is trivial.  

On the other hand, if $x =  -(2 m+1)^2$, then we find that 
\[
\operatorname{ord}_x ( \langle f_k, f_k \rangle)  =  \operatorname{ord}_x ( \varphi_k)  =
	\begin{cases}
	1 & |k| > 2m  \\ 0 &  |k |\le 2m  ,
	\end{cases}
\]
and as 
  a result 
  \[
 \mathcal{F}|_{x}^n  =
	\begin{cases}
  \mathcal{F}\vert _x & n \ge 0\\
    ( \cdots\oplus  \C_{-2m-4} \oplus  \C_{-2m-2} )  \,   \oplus \, (  \C_{2m+2} \oplus \C_{2m+ 4} \oplus \cdots )  & n =1\\
0  &  n >1  
	\end{cases}
\]
So if $x =  -(2 m+1)^2$, then the Jantzen filtration is nontrivial. The set of distinguished points from Section~\ref{subsec-distinguished} is therefore precisely this set of values.  (These are also the $x$ for which the   fibers $\mathcal{F}\vert _x$ are reducible, as we mentioned in Section~\ref{subsec-distinguished}.)

Let us now discuss infinitesimal unitarity. The scalars $\varphi_k(x)$ are positive if $x\ge 0$, and from this it follows that the  nonzero Jantzen quotient is infinitesimally unitary for these values.  If $x$ is negative but not of the form  $x =  -(2 m+1)^2$, then all    the $\varphi_k(x)$ are real and nonzero, but they are not all of the same sign, so the quotient is not infinitesimally unitary. The  nonzero Jantzen quotients at $x=-(2m+1)^2 $ are 
\[
 \mathcal{F}|_{x}^n  /\mathcal{F}|_{x}^{n+1}\cong 
	\begin{cases}
  \C_{-2m} \oplus \C_{-2 m +2} \oplus \cdots  \oplus \C_{2m} & n =0\\
    ( \cdots\oplus  \C_{-2m-4} \oplus  \C_{-2m-2} )  \,   \oplus \, (  \C_{2m+2} \oplus \C_{2m+ 4} \oplus \cdots )  & n =1 .\\
	\end{cases}
\]
The $n=0$ quotient is infinitesimally unitary; the other one is not.

 This completes our computations  of the Jantzen quotients.  In summary we find that: 

\begin{enumerate}[\rm (a)]

\item For $x>0$ the     unitary representation  of $\boldsymbol{G}|^{\sigma}_{x}\cong SU(1,1)$ on the (completion of the)  unique  infinitesimally unitary  Jantzen quotient  is the unitary spherical principal series representation with  Casimir $-(x +{1})/x$.
As $x$ varies these representations   exhaust    the spherical unitary principal series except for the base of the spherical principal series.\footnote{In a certain sense, the base of the spherical unitary principal series is located at $x=\infty$.  This can be made precise using the ``deformation to the normal cone'' family of Harish-Chandra pairs over the projective line considered in \cite[2.3.1]{Ber2016}.}

\item The   unitary representation  of the motion group $\boldsymbol{G}|^{\sigma}_{0} $  on   the unique and infinitesimally unitary Jantzen quotient   is the unique spherical unitary irreducible representation $\pi_{-1}$ on which the Casimir element $4EF$, which generates the center of the enveloping algebra,  acts as multiplication by $-1$.  

\item The remaining infinitesimally unitary Jantzen quotients only occur at    $x = -(2m+1)^{-2}$, with $m=0,1,2\dots$.  There is a unique one at each such $x$, and  the  unitary representation of $\boldsymbol{G}|^{\sigma}_{x}\cong SU(2)$  on it is the irreducible representation of highest weight $2m$. These representations   exhaust the  unitary irreducible representations of $SU(2)$ that contain the   $0$-weight of $U(1)$.
 \end{enumerate}
 Comparing the above to the contraction families we find that 
 \begin{enumerate}[\rm (a)]
 \setcounter{enumi}{3}
\item The unitary representations for $x> 0$ are those that appear in the contraction of the unitary principal series  $SU(1,1)$ to $\pi_{-1}$ \cite{Dooley85,Subag12} (also compare \cite[sec. 9.2.4]{Vil2}), while the discretely  occurring unitary  representations in the region $x<0$   are those that appear in the contraction  of  unitary irreducible  representations of $SU(2)$  to $\pi_{-1}$ \cite{Inonu-Wigner53,Dooley83,Subag12}. 
 \end{enumerate}

 \begin{remark} Although we focused above on a family that is relevant to the theory of  contractions of representations, even for the $SU(1,1)$ case we have studied here there are other families that   show     quite different features, and may be of interest for other purposes. For instance if we analyze  the family with Casimir 
 \[
 c_{\mathcal F}(z)=\frac{1-z}{z}
 \]
 in the same way, then we obtain infinitesimally unitary quotients that include both discrete series and complementary series representations when $x>0$. See Figure~\ref{fig:complementary-series-family}. 
\end{remark}

 \begin{figure}[htb] 
  \centering
  \includegraphics[scale=0.20]{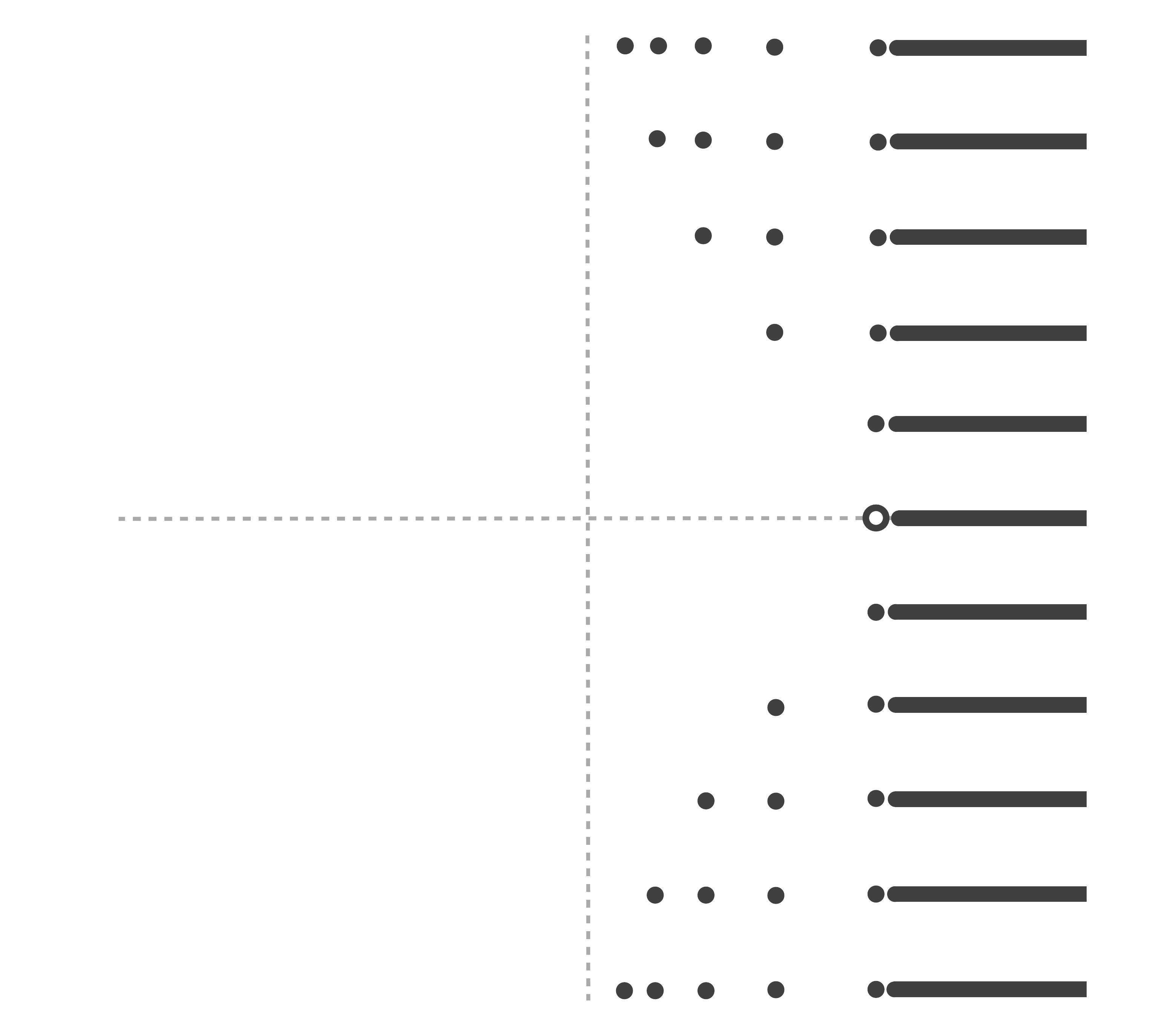} 
  \captionsetup{width=0.8\textwidth}
  \caption{\small The  infinitesimally unitary Jantzen quotients for the family with Casimir $\frac{1-x}{x}$.    There are no infinitesimally unitary quotients in the region $x\le 0$. In the region $x>0$ there arise  all the discrete series with even weights, the trivial representation (indicated by a white dot),  and all the complementary series.}
  \label{fig:complementary-series-family}
\end{figure}


\bibliography{references}

\def\cprime{$'$}
\begin{thebibliography}{AvLTV15}

\bibitem[Ada17]{Adams2017}
J.~Adams.
\newblock Deforming representations of {SL(2,R)}.
\newblock Preprint, 2017.
\newblock \href{https://arxiv.org/abs/1701.05879}{arXiv:1701.05879}.

\bibitem[Afg15]{Afgoustidis15}
A.~Afgoustidis.
\newblock How tempered representations of a semisimple {L}ie group contract to
  its {C}artan motion group.
\newblock Preprint, 2015.
\newblock \href{https://arxiv.org/abs/1510.02650}{arXiv:1510.02650}.

\bibitem[AvLTV15]{AdamsEtAl2015}
J.~Adams, M.~van Leeuven, P.~Trapa, and D.~Vogan.
\newblock Unitary representations of real reductive groups.
\newblock Preprint, 2015.
\newblock \href{https://arxiv.org/abs/1212.2192v4}{arXiv:1212.2192}.

\bibitem[BHS16]{Ber2016}
J.~Bernstein, N.~Higson, and E.~M. Subag.
\newblock Algebraic families of {H}arish-{C}handra pairs.
\newblock Preprint, 2016.
\newblock \href{https://arxiv.org/abs/1610.03435}{arXiv:1610.03435}.

\bibitem[Bor91]{Borel1991}
A.~Borel.
\newblock {\em Linear algebraic groups}, volume 126 of {\em Graduate Texts in
  Mathematics}.
\newblock Springer-Verlag, New York, second edition, 1991.

\bibitem[DR83]{Dooley83}
A.~H. Dooley and J.~W. Rice.
\newblock Contractions of rotation groups and their representations.
\newblock {\em Math. Proc. Cambridge Philos. Soc.}, 94(3):509--517, 1983.

\bibitem[DR85]{Dooley85}
A.~H. Dooley and J.~W. Rice.
\newblock On contractions of semisimple {L}ie groups.
\newblock {\em Trans. Amer. Math. Soc.}, 289(1):185--202, 1985.

\bibitem[Hig08]{Higson08}
N.~Higson.
\newblock The {M}ackey analogy and {$K$}-theory.
\newblock In {\em Group representations, ergodic theory, and mathematical
  physics: a tribute to {G}eorge {W}. {M}ackey}, volume 449 of {\em Contemp.
  Math.}, pages 149--172. Amer. Math. Soc., Providence, RI, 2008.

\bibitem[Hig11]{Higson11}
N.~Higson.
\newblock On the analogy between complex semisimple groups and their {C}artan
  motion groups.
\newblock In {\em Noncommutative geometry and global analysis}, volume 546 of
  {\em Contemp. Math.}, pages 137--170. Amer. Math. Soc., Providence, RI, 2011.

\bibitem[Hum08]{Humphreys08}
J.~E. Humphreys.
\newblock {\em Representations of semisimple {L}ie algebras in the {BGG}
  category {$\scr{O}$}}, volume~94 of {\em Graduate Studies in Mathematics}.
\newblock American Mathematical Society, Providence, RI, 2008.

\bibitem[IW53]{Inonu-Wigner53}
E.~Inonu and E.~P. Wigner.
\newblock On the contraction of groups and their representations.
\newblock {\em Proc. Nat. Acad. Sci. U. S. A.}, 39:510--524, 1953.

\bibitem[Jan79]{Jantzen79}
J.~C. Jantzen.
\newblock {\em Moduln mit einem h\"ochsten {G}ewicht}, volume 750 of {\em
  Lecture Notes in Mathematics}.
\newblock Springer, Berlin, 1979.

\bibitem[Mac75]{Mackey75}
G.~W. Mackey.
\newblock On the analogy between semisimple {L}ie groups and certain related
  semi-direct product groups.
\newblock In {\em Lie groups and their representations ({P}roc. {S}ummer
  {S}chool, {B}olyai {J}\'anos {M}ath. {S}oc., {B}udapest, 1971)}, pages
  339--363. Halsted, New York, 1975.

\bibitem[Nel59]{Nelson59}
E.~Nelson.
\newblock Analytic vectors.
\newblock {\em Ann. of Math. (2)}, 70:572--615, 1959.

\bibitem[SBBM12]{Subag12}
E.~M. Subag, E.~M. Baruch, J.~L. Birman, and A.~Mann.
\newblock Strong contraction of the representations of the three-dimensional
  {L}ie algebras.
\newblock {\em J. Phys. A}, 45(26):265206, 2012.

\bibitem[Seg51]{Segal51}
I.~E. Segal.
\newblock A class of operator algebras which are determined by groups.
\newblock {\em Duke Math. J.}, 18:221--265, 1951.

\bibitem[Spr98]{Springer}
T.~A. Springer.
\newblock {\em Linear algebraic groups}, volume~9 of {\em Progress in
  Mathematics}.
\newblock Birkh\"auser Boston, Inc., Boston, MA, second edition, 1998.

\bibitem[VK93]{Vil2}
N.~Ja. Vilenkin and A.~U. Klimyk.
\newblock {\em Representation of {L}ie groups and special functions. {V}ol. 2},
  volume~74 of {\em Mathematics and its Applications (Soviet Series)}.
\newblock Kluwer Academic Publishers Group, Dordrecht, 1993.
\newblock Class I representations, special functions, and integral transforms,
  Translated from the Russian by V. A. Groza and A. A. Groza.

\end{thebibliography}
\bibliographystyle{alpha}

\end{document}